\documentclass[a4paper,12pt]{article}

\usepackage[english]{babel}
\usepackage{amsmath,amsfonts,amssymb,amsthm,bbm}

\usepackage[top=2cm, bottom=2cm, left=2.5cm, right=2.5cm]{geometry}

\usepackage{indentfirst}
\usepackage[T1]{fontenc} 

\usepackage{graphicx}
\usepackage{float}

\usepackage{epsfig}
\usepackage{epstopdf}

\newtheorem{theorem}{Theorem}[section] 
\newtheorem{lemma}[theorem]{Lemma}     
\newtheorem{definition}[theorem]{Definition}
\newtheorem{corollary}[theorem]{Corollary}

\numberwithin{equation}{section}

\renewcommand{\H}{\mathbb H}
\newcommand{\N}{\mathbb N}

\newcommand{\R}{\mathbb R}
\newcommand{\norml}{\mathcal{N}}

\newcommand{\Fi}{\mathbf \Phi}
\newcommand{\Ffi}{\mathbf \phi}
\newcommand{\I}{\mathbf{Id}}
\newcommand{\x}{\mathbbm x}
\newcommand{\y}{\mathbbm y}
\newcommand{\z}{\mathbbm z}
\newcommand{\e}{\mathbbm e}
\newcommand{\h}{\mathbbm h}
\newcommand{\C}{\mathbbm c}
\newcommand{\0}{\mathbf 0}

\newcommand*{\ip}[1]{\left\langle{#1}\right\rangle} 
\newcommand*{\norm}[1]{\left\lVert{#1}\right\rVert} 
\newcommand*{\conj}[1]{\overline{#1}} 
\newcommand*{\herm}[1]{#1^*} 

\renewcommand{\i}{\mathbf i}
\renewcommand{\j}{\mathbf j}
\renewcommand{\k}{\mathbf k}

\DeclareMathOperator*{\argmin}{arg\,min}

\newcommand{\supp}{\mathrm {supp}}

\begin{document}

\begin{center}
\Large
\textbf{Compressed sensing in the quaternion algebra}\\
\end{center}

\begin{center}
\begin{tabular}{cc}
	\textbf{Agnieszka Bade\'nska\textsuperscript{1}} & \textbf{{\L}ukasz B{\l}aszczyk\textsuperscript{1,2}} \\
	badenska@mini.pw.edu.pl & l.blaszczyk@mini.pw.edu.pl 
\medskip \\
	\textsuperscript{1} Faculty of Mathematics  & \textsuperscript{2} Institute of Radioelectronics \\ and Information Science & and Multimedia Technology \\
	Warsaw University of Technology & Warsaw University of Technology \\
	ul. Koszykowa 75 & ul. Nowowiejska 15/19 \\
	00-662 Warszawa, Poland & 00-665 Warszawa, Poland
\end{tabular}
\end{center}

\bigskip\noindent
\textbf{Keywords}: compressed sensing, quaternion, restricted isometry property, sparse signals.

\bigskip
\normalsize
\begin{abstract}
The~article concerns compressed sensing methods in the~quaternion algebra. We prove that it is possible to uniquely reconstruct -- by $\ell_1$-norm minimization -- a~sparse quaternion signal from a~limited number of its linear measurements, provided the~quaternion measurement matrix satisfies so-called restricted isometry property with a~sufficiently small constant. We also provide error estimates for the~reconstruction of a~non-sparse quaternion signal in the~noisy and noiseless cases.
\end{abstract}

\section{Introduction}

E.~Cand\'es et al. showed that -- in the~real or complex setting -- if a~measurement matrix satisfies so-called \textit{restricted isometry property} (RIP) with a~sufficiently small constant, then every sparse signal can be uniquely reconstructed from a~limited number of its linear measurements as a~solution of a~convex program of $\ell_1$-norm minimization (see~e.g.~\cite{cRIP,crt,ct} and \cite{introCS} for more references). Sparsity of the~signal is a~natural assumption -- most of well known signals have a~sparse representation in an~appropriate basis (e.g. wavelet representation of an~image). Moreover, if the~original signal was not sparse, the~same minimization procedure provides a~good sparse approximation of the~signal and the~procedure is stable in the~sense that the~error is bounded above by the~$\ell_1$-norm of the~difference between the~original signal and its best sparse approximation.

For a~certain time the~attention of the~researchers in the~theory of \textit{compressed sensing} has mostly been focused on real and complex signals. Over the~last decade there have been published results of numerical experiments suggesting that the~compressed sensing methods can be successfully applied also in the~quaternion algebra~\cite{barthelemy2015,hawes2014,l1minqs}, however, until recently there were no theoretical results that could explain the~success of these experiments. The~aim of our research is to develop theoretical background of the~compressed sensing theory in the~quaternion algebra. 

Our~first step towards this goal was proving that one can uniquely reconstruct a~sparse quaternion signal -- by $\ell_1$-norm minimization -- provided the~real measurement matrix satisfies the~RIP (for quaternion vectors) with a~sufficiently small constant (\cite[Corrolary~5.1]{bb}). This result can be directly applied since any~real matrix satisfying the~RIP for real vectors, satisfies the~RIP for quaternion vectors with the~same constant ({\cite[Lemma~3.2]{bb}}).
We also want to point out a~very interesting recent result of N.~Gomes, S.~Hartmann and U.~K\"ahler concerning the~quaternion Fourier matrices -- arising in colour representation of images. They showed that with high probability such matrices allow a~sparse reconstruction by means of the~$\ell_1$-minimization {\cite[Theorem~3.2]{ghk}}. Their proof, however, is straightforward and does not use the~notion of~RIP.

The generalization of compressed sensing to the~quaternion algebra would be significant due to their wide applications. Apart from the~classical applications (in quantum mechanics and for the~description of 3D solid body rotations), quaternions have also been used in 3D and 4D signal processing \cite{snopek2015}, in particular to represent colour images (e.g.~in the RGB or CMYK models). Due to the~extension of classical tools (like the~Fourier transform~\cite{ell2014}) to the~quaternion algebra it is possible to investigate colour images without need of treating each component separately~\cite{dubey2014,ell2014}. That is why quaternions have found numerous applications in image filtering, image enhancement, pattern recognition, edge detection and watermarking \cite{es,G1,khalil2012,P1,rzadkowski2015,T1,W1}. There has also been proposed a~dual-tree quaternion wavelet transform in a~multiscale analysis of geometric image features~\cite{ccb}. For this purpose an~alternative representation of quaternions is used -- through its magnitude (norm) and three phase angles: two of them encode phase shifts while the~third contains image texture information \cite{bulow2001}. 
In view of numerous articles presenting results of numerical experiments of quaternion signal processing and their possible applications, there is a~natural need of further thorough theoretical investigations in this field.

In this article we extend the~fundamental result of the~compressed sensing theory to the~quaternion case, namely we show that if a~quaternion measurement matrix satisfies the~RIP with a~sufficiently small constant, then it is possible to reconstruct sparse quaternion signals from a~small number of their measurements via $\ell_1$-norm minimization (Corollary~\ref{l1mincor}). We also estimate the~error of reconstruction of a~non-sparse signal from exact and noisy data (Theorem~\ref{l1minthm}). Note that these results not only generalize the~previous ones {\cite[Theorem~4.1, Corrolary~5.1]{bb}} but also improve them by decreasing the~error estimation's constants. This enhancement was possible due to using algebraic properties of quaternion Hermitian matrices (Lemma~\ref{herm-norm}) to derive characterization of the~restricted isometry constants (Lemma~\ref{RIPequiv}) analogous to the~real and complex case. Consequently, one can carefully follow steps of the~classical Cand\'es' proof~\cite{cRIP} with caution to the~non-commutativity of quaternion multiplication.

It is known that e.g. real Gaussian and Bernoulli random matrices, also partial Discrete Fourier Transform matrices satisfy the~RIP (with overwhelming probability)~\cite{introCS}, however, until recently there were no known examples of quaternion matrices satisfying this condition. It has been believed that quaternion Gaussian random matrices satisfy RIP and, therefore, they have been widely used in numerical experiments \cite{barthelemy2015,hawes2014,l1minqs} but there was a~lack of theoretical justification of this conviction. 
In the~subsequent article~\cite{bbRIP} we prove that this hypothesis is true, i.e. quaternion Gaussian matrices satisfy the~RIP, and we provide estimates on matrix sizes that guarantee the~RIP with overwhelming probability. This result, together with the~main results of this article (Theorem~\ref{l1minthm}, Corollary~\ref{l1mincor}), constitute the~theoretical foundation of the~classical compressed sensing methods in the~quaternion algebra.

The~article is organized as follows. First, we recall basic notation and facts concerning the~quaternion algebra with particular emphasis put on the~properties of Hermitian form and Hermitian matrices. The~third section is devoted to the~RIP and characterization of the~restricted isometry constants in terms of Hermitian matrix norm. The~fourth and fifth sections contain proofs of the~main results of the~article. In the~sixth section we present results of numerical experiments illustrating our results -- we may see, in particular, that the~rate of perfect reconstructions in the~quaternion case is higher than in the~real case experiment with the~same parameters. We conclude with a~short r\'esum\'e of the~obtained results and our further research perspectives.

\section{Algebra of quaternions}

Denote by~$\H$ the~algebra of quaternions
$$ q=a+b\i+c\j+d\k, \quad \textrm{where}\quad a,b,c,d\in\R, $$
endowed with the~standard norm
$$ |q|=\sqrt{q\conj{q}}=\sqrt{a^2+b^2+c^2+d^2}, $$
where $\conj{q}=a-b\i-c\j-d\k$ is the~conjugate of~$q$. 

Recall that multiplication is associative but in general not commutative in the~quaternion algebra and is defined by the~following rules
$$ \i^2=\j^2=\k^2=\i\j\k=-1 $$
and
$$ \i\j=-\j\i=\k, \quad \j\k=-\k\j=\i, \quad \k\i=-\i\k=\j. $$
Multiplication is distributive with respect to addition and has a~neutral element $1\in\H$, hence $\H$ forms a~ring, which is usually called a~noncommutative field.
We also have the~property that
$$ \conj{q\cdot w}=\conj{w}\cdot\conj{q} \quad\textrm{for any}\quad q,w\in\H. $$

In what follows we will interpret signals as vectors with quaternion coordinates, i.e. elements of $\H^n$. Algebraically $\H^n$ is a~module over the~ring $\H$, usually called the~quaternion vector space. We will also consider matrices with quaternion entries with usual multiplication rules.

For any matrix $\Fi\in\H^{m\times n}$ with quaternion entries by $\herm{\Fi}$ we denote the~adjoint matrix, i.e. $\herm{\Fi}=\conj{\Fi}^T$, where $T$ is the~transpose. The~same notation applies to quaternion vectors $\x\in\H^n$ which can be interpreted as one-column matrices $\x\in\H^{n\times1}$. Obviously $\herm{\left(\herm{\Fi}\right)}=\Fi$.

A~matrix $\Fi\in\H^{m\times n}$ defines a~$\H$-linear transformation $\Fi:\H^n\to\H^m$ (in terms of the~right quaternion vector space, i.e. considering the~right scalar multiplication) which acts by the~standard matrix-vector multiplication:
$$ \Fi(\x+\y)=\Fi\x+\Fi\y \quad \textrm{and} \quad \Fi(\x q)=(\Fi\x)q \quad \textrm{for any}\quad \x,\y\in\H^n,\; q\in\H. $$
We also have that
$$ \herm{(\Fi q)}=\conj{q}\herm{\Fi}, \quad \herm{(q\Fi)}=\herm{\Fi}\conj{q}, \quad \herm{(\Fi\x)}=\herm{\x}\herm{\Fi}, \quad \herm{(\Fi\mathbf{\Psi})}=\herm{\mathbf{\Psi}}\herm{\Fi} $$
for all $\Fi\in\H^{m\times n}$, $q\in\H$, $\x\in\H^n$, $\mathbf{\Psi}\in\H^{n\times p}$.

For any $n\in\N$ we introduce the~following Hermitian form $\ip{\cdot,\cdot}\colon\H^n\times\H^n\to\H$ with quaternion values:
$$ \ip{\x,\y}=\herm{\y}\x=\sum_{i=1}^{n}\conj{y_i}x_i, \quad\textrm{where}\quad \x=(x_1,\ldots,x_n)^T,\; \y=(y_1,\ldots,y_n)^T\in\H^n $$
and $T$ is the~transpose. Denote also
$$ \norm{\x}_2=\sqrt{\ip{\x,\x}}=\sqrt{\sum_{i=1}^{n}|x_i|^2}, \quad \textrm{for any }\x=(x_1,\ldots,x_n)^T\in\H^n. $$

It is straightforward calculation to verify that $\ip{\cdot,\cdot}$ satisfies the~following properties of an~inner product for all $\x,\y,\z\in\H^n$ and $q\in\H$.
	\begin{itemize}
		\item $\displaystyle \conj{\ip{\x,\y}}=\ip{\y,\x}$.
		\item $\displaystyle \ip{\x q,\y}=\ip{\x,\y}q$.
		\item $\displaystyle \ip{\x+\y,\z}=\ip{\x,\z}+\ip{\y,\z}$.
		\item $\displaystyle \ip{\x,\x}=\norm{\x}_2^2\geq0$ and $\displaystyle \norm{\x}_2^2=0 \quad \iff \quad \x=\0$.
	\end{itemize}
Hence $\norm{\cdot}_2$ satisfies the axioms of a~norm in~$\H^n$.

By carefully following the~classical steps of the~proof we also get the~Cauchy-Schwarz inequality (cf.\cite[Lemma 2.2]{bb}).
	$$ |\ip{\x,\y}|\leq\norm{\x}_2\cdot\norm{\y}_2 $$
for any $\x,\y\in\H^n$. 

Notice that for $\Fi\in\H^{m\times n}$ the~matrix $\herm{\Fi}$ defines the~adjoint $\H$-linear transformation since
$$ \ip{\x,\herm{\Fi}\y}=\herm{\left(\herm{\Fi}\y\right)}\x=\herm{\y}\Fi\x=\ip{\Fi\x,\y} \quad \textrm{for} \quad \x\in\H^n,\y\in\H^m. $$
Recall also that a~linear transformation (matrix) $\mathbf{\Psi}\in\H^{n\times n}$ is called \textit{Hermitian} if $\herm{\mathbf{\Psi}}=\mathbf{\Psi}$. Obviously $\herm{\Fi}\Fi$ is Hermitian for any $\Fi\in\H^{m\times n}$. 

In the~next section we will use the~following property of Hermitian matrices.
\begin{lemma}\label{herm-norm}
Suppose $\mathbf{\Psi}\in\H^{n\times n}$ is Hermitian. Then
$$ \norm{\mathbf{\Psi}}_{2\to2}=\max_{\x\in\H^n,\norm{\x}_2=1}\left|\ip{\mathbf{\Psi}\x,\x}\right|=\max_{\x\in\H^n\setminus\{\0\}}\frac{\left|\ip{\mathbf{\Psi}\x,\x}\right|}{\norm{\x}_2^2}, $$
where $\norm{\cdot}_{2\to2}$ is the~standard operator norm in the~right quaternion vector space $\H^n$ endowed with the~norm $\norm{\cdot}_2$, i.e.
$$ \norm{\mathbf{\Psi}}_{2\to2}=\max_{\x\in\H^n\setminus\{\0\}}\frac{\norm{\mathbf{\Psi}\x}_2}{\norm{\x}_2}= \max_{\x\in\H^n,\norm{\x}_2=1}\norm{\mathbf{\Psi}\x}_2. $$
\end{lemma}
\begin{proof}
Recall that a~Hermitian matrix has real (right) eigenvalues~\cite{qalgebra}. Moreover, there exists an~orthonormal (in terms of the~$\H$-linear form $\ip{\cdot,\cdot}$) base of $\H^n$ consisting of eigenvectors $\x_i$ corresponding to eigenvalues $\lambda_i\in\R$, $i=1,\ldots,n$ (cf.{\cite[Theorem 5.3.6. (c)]{qalgebra}}), i.e.
$$ \mathbf{\Psi}\x_i=\x_i\lambda_i\quad \textrm{and}\quad \ip{\x_i,\x_j}=\herm{\x_j}\x_i=\delta_{i,j}\quad \textrm{for}\quad i,j=1,\ldots,n. $$
Denote $\lambda_{\max}=\max_i|\lambda_i|$. Then $\norm{\mathbf{\Psi}}_{2\to2}=\lambda_{\max}$. Indeed, for any $\x=\sum\limits_{i=1}^n\x_i\alpha_i\in\H^n$ with $\norm{\x}_2^2=\sum\limits_{i=1}^n|\alpha_i|^2=1$, since $\x_i$ are orthonormal, we have
$$ \norm{\mathbf{\Psi}\x}_2^2=\norm{\sum_{i=1}^n\mathbf{\Psi}\x_i\alpha_i}_2^2=\norm{\sum_{i=1}^n\x_i\lambda_i\alpha_i}_2^2=\sum_{i=1}^n\left|\lambda_i\alpha_i\right|^2 \leq\lambda_{\max}\underbrace{\sum_{i=1}^n\left|\alpha_i\right|^2}_{=1} $$
and for the~appropriate eigenvector $\x_i$ for which $|\lambda_i|=\lambda_{\max}$, 
$$ \norm{\mathbf{\Psi}\x_i}_2=\norm{\x_i\lambda_i}_2=\lambda_{\max}\norm{\x_i}_2=\lambda_{\max}. $$
On the~other hand, since $\lambda_i$ are real,
$$\aligned
   \ip{\mathbf{\Psi}\x,\x} &= \herm{\x}\mathbf{\Psi}\x=\herm{\left(\sum_{i=1}^n\x_i\alpha_i\right)}\left(\sum_{j=1}^n\mathbf{\Psi}\x_j\alpha_j\right)= \left(\sum_{i=1}^n\conj{\alpha_i}\herm{\x_i}\right)\left(\sum_{j=1}^n\x_j\lambda_j\alpha_j\right)\\
   & = \sum_{i=1}^n\conj{\alpha_i}\lambda_i\alpha_i \stackrel{\lambda_i\in\R}{=} \sum_{i=1}^n\lambda_i|\alpha_i|^2. 
 \endaligned $$
Hence 
$$ \left|\ip{\mathbf{\Psi}\x,\x}\right|\leq \lambda_{\max}\sum\limits_{i=1}^n|\alpha_i|^2=\lambda_{\max} $$ 
and -- again -- for the~appropriate eigenvector the~last two quantities are equal. The~result follows.
\end{proof}

In what follows we will consider $\norm{\cdot}_p$ norms for quaternion vectors $\x\in\H^n$ defined in the~standard way:
$$ \norm{\x}_{p}=\left(\sum_{i=1}^{n}|x_i|^p\right)^{1/p}, \quad \textrm{for}\quad p\in[1,\infty) $$
and
$$ \norm{\x}_{\infty}=\max_{1\leq i\leq n}|x_i|, $$
where $\x=(x_1,\ldots,x_n)^T$.
We will also apply the~usual notation for the~cardinality of the~support of~$\x$, i.e.
$$ \norm{\x}_0=\#\supp(\x), \quad\textrm{where}\quad \supp(\x)=\{i\in\{1,\ldots,n\}\colon x_i\neq0\}. $$

\section{Restricted Isometry Property}

Recall that we call a~vector (signal) $\x\in\H^n$ $s$-sparse if it has at most $s$ nonzero coordinates, i.e.
$$ \norm{\x}_0\leq s. $$
As it was mentioned in the~introduction, one of the~conditions which guarantees exact reconstruction of a~sparse real signal from a~few number of its linear measurements is that the~measurement matrix satisfies so-called restricted isometry property (RIP) with a~sufficiently small constant. The~notion of restricted isometry constants was introduced by Cand\`es and Tao in~\cite{ct}. Here we generalize it to quaternion signals.

\begin{definition}\label{RIP}
Let $s\in\N$ and $\Fi\in\H^{m\times n}$. We say that $\Fi$  satisfies the $s$-restricted isometry property (for quaternion vectors) with a~constant $\delta_s\geq0$ if
\begin{equation}\label{qRIP}
	 \left(1-\delta_s\right)\norm{\x}_2^2\leq\norm{\Fi\x}_2^2\leq\left(1+\delta_s\right)\norm{\x}_2^2 
\end{equation}
for all $s$-sparse quaternion vectors $\x\in\H^n$. The~smallest number $\delta_s\geq0$ with this property is called the $s$-restricted isometry constant. 
\end{definition}
Note that we can define $s$-restricted isometry constants for any matrix $\Fi\in\H^{m\times n}$ and any number $s\in\{1,\ldots,n\}$. It has been proved that if a~real matrix $\Fi\in\R^{m\times n}$ satisfies the~inequality~\eqref{qRIP} for real $s$-sparse vectors~$\x\in\R^n$, then it also satisfies it -- with the~same constant~$\delta_s$ -- for $s$-sparse quaternion vectors~$\x\in\H^n$ \cite[Lemma~3.2]{bb}. 

The~following lemma extends an~analogous result, known for real and complex matrices~\cite{introCS}, to the~quaternion case. As is it accustomed, for a~matrix $\Fi\in\H^{m\times n}$ and a~set of indices $S\subset\{1,\ldots,n\}$ with $\# S=s$ by $\Fi_S\in\H^{m\times s}$ we denote the~matrix consisting of columns of~$\Fi$ with indices in the~set~$S$.
\begin{lemma}\label{RIPequiv}
The $s$-restricted isometry constant $\delta_s$ of a~matrix $\Fi\in\H^{m\times n}$ equivalently equals
$$ \delta_s=\max_{S\subset\{1,\ldots,n\},\#S\leq s}\norm{\herm{\Fi_S}\Fi_S-\I}_{2\to2}. $$
\end{lemma}
\begin{proof}
We proceed as in~{\cite[Chapter 6]{introCS}}. Fix any $s\in\{1,\ldots,n\}$ and $S\subset\{1,\ldots,n\}$ with $\#S\leq s$. Notice that the~condition \eqref{qRIP} can be equivalently rewritten as
$$ \left|\norm{\Fi_S\x}_2^2-\norm{\x}_2^2\right|\leq\delta_s\norm{\x}_2^2 \quad \textrm{for all} \quad \x\in\H^s, $$
where $\delta_s$ is the $s$-restricted isometry constant of~$\Fi$.
The left hand side equals
$$ \left|\norm{\Fi_S\x}_2^2-\norm{\x}_2^2\right|\leq\delta_s\norm{\x}_2^2 = \left|\ip{\Fi_S\x,\Fi_S\x}-\ip{\x,\x}\right|=\left|\ip{\left(\herm{\Fi_S}\Fi_S-\I\right)\x,\x}\right|$$
and by the~Lemma~\ref{herm-norm}, since the~matrix $\herm{\Fi_S}\Fi_S-\I$ is Hermitian, we get that
$$ \max_{\x\in\H^s\setminus\{\0\}}\frac{\left|\ip{\left(\herm{\Fi_S}\Fi_S-\I\right)\x,\x}\right|}{\norm{\x}_2^2} = \norm{\herm{\Fi_S}\Fi_S-\I}_{2\to2}. $$
Arbitrary choice of $s$ and $S$ finishes the proof.
\end{proof}

The~next result is an~important tool in the~proof of Theorem~\ref{l1minthm}. Having the above equivalence and the~Cauchy-Schwarz inequality we are able to obtain for quaternion vectors the~same estimate as in the~real and complex case (cf. {\cite[Lemma~2.1]{cRIP}} and {\cite[Proposition~6.3]{introCS}}). 

\begin{lemma}\label{RIPip}
 Let $\delta_s$ be the~$s$-restricted isometry constant for a~matrix $\Fi\in\H^{m\times n}$ for $s\in\{1,\ldots,n\}$. For any pair of $\x,\y\in\H^n$ with disjoint supports and such that $\norm{\x}_0\leq s_1$ and $\norm{\y}_0\leq s_2$, where $s_1+s_2\leq n$, we have that 
$$  \left|\ip{\Fi\x,\Fi\y}\right|\leq \delta_{s_1+s_2}\norm{\x}_2\norm{\y}_2. $$
\end{lemma}
\begin{proof} 
In this proof we will use the~following notation: for any~vector $\x\in\H^n$ and a~set of indices $S\subset\{1,\ldots,n\}$ with $\# S=s$ by $\x_{|S}\in\H^s$ we denote the~vector of $\x$-coordinates with indices in~$S$.

Take any vectors $\x,\y\in\H^n$ satisfying the~assumptions of the~lemma and denote $S=\supp(\x)\cup\supp(\y)$. Obviously $\#S=s_1+s_2$. Since $\x$ and $\y$ have disjoint supports, they are orthogonal, i.e. $\ip{\x,\y}=\ip{\x_{|S},\y_{|S}}=0$. Using the~Cauchy-Schwarz inequality and Lemma~\ref{RIPequiv} we get that
$$\aligned
		 \left|\ip{\Fi\x,\Fi\y}\right| &= \left|\ip{\Fi_S\x_{|S},\Fi_S\y_{|S}}-\ip{\x_{|S},\y_{|S}}\right| = \left|\ip{\left(\herm{\Fi_S}\Fi_S-\I\right)\x_{|S},\y_{|S}}\right| \\
		 															 &\leq \norm{\herm{\Fi_S}\Fi_S-\I}_{2\to2}\norm{\x_{|S}}_2\norm{\y_{|S}}_2\leq \delta_{s_1+s_2}\norm{\x_{|S}}_2\norm{\y_{|S}}_2, 
\endaligned $$
which finishes the~proof since $\norm{\x_{|S}}_2=\norm{\x}_2$ and $\norm{\y_{|S}}_2=\norm{\y}_2$.
\end{proof}

\section{Stable reconstruction from noisy data}  

As we mentioned in the~introduction, our aim is to reconstruct a~quaternion signal from a~limited number of its linear measurements with quaternion coefficients. We will also assume the~presence of a~white noise with bounded $\ell_2$ quaternion norm. The~observables are, therefore, given by
$$ \y=\Fi\x+\e, \quad\textrm{where}\quad \x\in\H^n, \; \Fi\in\H^{m\times{n}}, \; \y\in\H^m \;\textrm{and}\; \e\in\H^n \textrm{ with } \norm{\e}_2\leq\eta $$
for some $m\leq n$ and $\eta\geq0$.

We will use the~following notation: for any $\h\in\H^n$ and a~set of indices $T\subset\{1,\ldots,n\}$, the~vector $\h_T\in\H^n$ is supported on~$T$ with entries
$$ \left(\h_T\right)_i=\left\{\begin{array}{cl} h_i & \textrm{if } i\in T, \\ 0 & \textrm{otherwise}, \end{array}\right. \quad \textrm{where} \quad \h=(h_1,\ldots,h_n)^T. $$
The~complement of $T\subset\{1,\ldots,n\}$ will be denoted by $T^c=\{1,\ldots,n\}\setminus{T}$ and the~symbol~$\x_s$ will be used for the~best $s$-sparse approximation of the~vector~$\x$.

The~following result is a~generalization of {\cite[Theorem~1.3]{cRIP}} and \cite[Theorem 4.1]{bb} to the~full quaternion case. It also improves the~error estimate's constants from~\cite[Theorem~4.1]{bb}.

\begin{theorem}\label{l1minthm}
Suppose that $\Fi\in\H^{m\times n}$ satisfies the~$2s$-restricted isometry property with a~constant $\delta_{2s}<\sqrt{2}-1$ and let $\eta\geq0$. Then, for any $\x\in\H^n$ and $\y=\Fi\x+\e$ with $\norm{\e}_2\leq\eta$, the~solution $\x^{\#}$ of the~problem
\begin{equation}\label{l1minproblem}
\argmin\limits_{\z\in\H^n} \norm{\z}_1 \quad\text{subject to}\quad\norm{\Fi\z-\y}_2\leq\eta
\end{equation}
satisfies 
\begin{equation} \label{main-ineq}
\norm{\x^{\#}-\x}_2 \leq \frac{C_0}{\sqrt{s}}\norm{\x-\x_s}_1 + C_1\eta
\end{equation} with constants 
$$ C_0 = 2\cdot \frac{1+\left(\sqrt{2}-1\right)\delta_{2s}}{1-\left(\sqrt{2}+1\right)\delta_{2s}} ,\quad C_1 = \frac{4\sqrt{1+\delta_{2s}}}{1-\left(\sqrt{2}+1\right)\delta_{2s}}, $$
where $\x_s$ denotes the~best $s$-sparse approximation of $\x$.
\end{theorem}

\begin{proof}
Denote 
$$ \h=\x^{\#}-\x $$
and decompose $\h$ into a~sum of vectors $\h_{T_0},\h_{T_1},\h_{T_2}\ldots$ in the~following way: let $T_0$ be the~set of indices of $\x$ coordinates with the~biggest quaternion norms (hence $\x_s=\x_{T_0}$); $T_1$ is the~set of indices of $\h_{T_0^c}$ coordinates with the~biggest norms, $T_1$ is the~set of indices of $\h_{\left(T_0\cup T_1\right)^c}$ coordinates with the~biggest norms, etc. Then obviously all $\h_{T_j}$ are $s$-sparse and have disjoint supports.
In what follows we will separately estimate norms $\norm{\h_{T_0\cup T_1}}_2$ and~$\norm{\h_{\left(T_0\cup T_1\right)^c}}_2$.

Notice that for $j\geq2$ we have that
$$ \norm{\h_{T_j}}_2^2=\sum_{i\in T_j}|h_i|^2\leq \sum_{i\in T_j}\norm{\h_{T_j}}_{\infty}^2\leq s\norm{\h_{T_j}}_{\infty}^2, $$
where $h_i$ are the~coordinates of~$\h$ and $\norm{\h_{T_j}}_{\infty}=\max\limits_{i\in{T_j}}|h_i|$ (the~last $T_j$ may have less than~$s$ nonzero coordinates).
Moreover, since all non-zero coordinates of~$\h_{T_{j-1}}$ have norms not smaller than non-zero coordinates of $\h_{T_j}$,
$$ \norm{\h_{T_j}}_{\infty}\leq\frac{1}{s}\sum_{i\in{T_{j-1}}}|h_i|=\frac{1}{s}\norm{\h_{T_{j-1}}}_1. $$
Hence, for $j\geq2$ we get that
$$ 	\norm{\h_{T_j}}_2\leq \sqrt{s}\norm{\h_{T_j}}_{\infty}\leq \frac{1}{\sqrt{s}}\norm{\h_{T_{j-1}}}_1, $$
which implies
\begin{equation}\label{sumnormhTj}
 \sum_{j\geq2}\norm{\h_{T_j}}_2\leq \frac{1}{\sqrt{s}}\sum_{j\geq1}\norm{\h_{T_j}}_1\leq \frac{1}{\sqrt{s}}\norm{\h_{T_0^c}}_1.
\end{equation}
Finally 
\begin{equation}\label{normhT01c}
	\norm{\h_{\left(T_0\cup T_1\right)^c}}_2=\norm{\sum_{j\geq2}\h_{T_j}}_2\leq \sum_{j\geq2}\norm{\h_{T_j}}_2\leq \frac{1}{\sqrt{s}}\norm{\h_{T_0^c}}_1.
\end{equation}

Observe that $\norm{\h_{T_0^c}}_1$ can not be to large. Indeed, since $\norm{\x^{\#}}_1=\norm{\x+\h}_1$ is minimal
$$ \norm{\x}_1\geq\norm{\x+\h}_1= \norm{\x_{T_0}+\h_{T_0}}_1+ \norm{\x_{T_0^c}+\h_{T_0^c}}_1 \geq \norm{\x_{T_0}}_1-\norm{\h_{T_0}}_1-\norm{\x_{T_0^c}}_1+\norm{\h_{T_0^c}}_1, $$
hence
$$ \norm{\x_{T_0^c}}_1=\norm{\x}_1-\norm{\x_{T_0}}_1\geq -\norm{\h_{T_0}}_1-\norm{\x_{T_0^c}}_1+\norm{\h_{T_0^c}}_1 $$
and therefore
\begin{equation}\label{normhT0c}
	\norm{\h_{T_0^c}}_1\leq \norm{\h_{T_0}}_1+2\norm{\x_{T_0^c}}_1.
\end{equation}

Now, the Cauchy-Schwarz inequality immediately implies that $\norm{\h_{T_0}}_1\leq\sqrt{s}\norm{\h_{T_0}}_2$.
From this, (\ref{normhT01c}) and (\ref{normhT0c}) we conclude that
\begin{equation}\label{ineq1} 
	\norm{\h_{\left(T_0\cup T_1\right)^c}}_2\leq \frac{1}{\sqrt{s}}\norm{\h_{T_0^c}}_1 \leq \frac{1}{\sqrt{s}}\norm{\h_{T_0}}_1+\frac{2}{\sqrt{s}}\norm{\x_{T_0^c}}_1\leq 
	\norm{\h_{T_0}}_2+2\epsilon,
\end{equation}
where $\epsilon=\frac{1}{\sqrt{s}}\norm{\x_{T_0^c}}_1=\frac{1}{\sqrt{s}}\norm{\x-\x_s}_1$. This is the~first ingredient of the~final estimate.

\medskip
Now, we are going to estimate the~remaining component, i.e. $\norm{\h_{T_0\cup T_1}}_2$.
Using $\H$-linearity of $\ip{\cdot,\cdot}$, we get that
$$\aligned
	\norm{\Fi\h_{T_0\cup T_1}}_2^2&=\ip{\Fi\h_{T_0\cup T_1},\Fi\h_{T_0\cup T_1}}=\ip{\Fi\h_{T_0\cup T_1},\Fi\h}-\sum_{j\geq2}\ip{\Fi\h_{T_0\cup T_1},\Fi\h_{T_j}}\\
		&=\ip{\Fi\h_{T_0\cup T_1},\Fi\h}-\sum_{j\geq2}\ip{\Fi\h_{T_0},\Fi\h_{T_j}}-\sum_{j\geq2}\ip{\Fi\h_{T_1},\Fi\h_{T_j}}.
\endaligned$$
Estimate of the~norm of the~first element follows from the~Cauchy-Schwarz inequality in the~quaternion case, RIP and the~following simple observation
$$ 	\norm{\Fi\left(\x^{\#}-\x\right)}_2\leq  \norm{\Fi\x^{\#}-\y}_2+\norm{\Fi\x-\y}_2\leq 2\eta, $$
which follows from the~fact that $\x^{\#}$ is the~minimizer of~(\ref{l1minproblem}) and $\x$ is feasible. We get therefore that
\begin{equation}\label{FihT01Fih}
	\left|\ip{\Fi\h_{T_0\cup T_1},\Fi\h}\right|\leq\norm{\Fi\h_{T_0\cup T_1}}_2\cdot\norm{\Fi\h}_2 \leq \sqrt{1+\delta_{2s}}\norm{\h_{T_0\cup T_1}}_2\cdot2\eta.
\end{equation}
For the~remaining terms recall that $\h_{T_j}$, $j\geq0$, are $s$-sparse with pairwise disjoint supports and apply Lemma~\ref{RIPip}
$$ \left|\ip{\Fi\h_{T_i},\Fi\h_{T_j}}\right|\leq\delta_{2s}\cdot\norm{\h_{T_i}}_2\cdot\norm{\h_{T_j}}_2 \quad\text{for}\quad i=0,1 \quad\text{and}\quad j\geq2, $$

Since $T_0$ and $T_1$ are disjoint, $\norm{\h_{T_0\cup T_1}}_2^2=\norm{\h_{T_0}}_2^2+\norm{\h_{T_1}}_2^2$ and therefore  
$\norm{\h_{T_0}}_2+\norm{\h_{T_1}}_2\leq \sqrt2\norm{\h_{T_0\cup T_1}}_2$. Hence, using the~RIP, (\ref{FihT01Fih}) and (\ref{sumnormhTj}), 
$$\aligned 
	\left(1-\delta_{2s}\right)\norm{\h_{T_0\cup T_1}}_2^2 &\leq \norm{\Fi\h_{T_0\cup T_1}}_2^2 \\
	&\leq \sqrt{1+\delta_{2s}}\norm{\h_{T_0\cup T_1}}_2\cdot2\eta+ \delta_{2s}\cdot\left(\norm{\h_{T_0}}_2+\norm{\h_{T_1}}_2\right)\sum_{j\geq2}\norm{\h_{T_j}}_2\\
	&\leq \left(2\sqrt{1+\delta_{2s}}\cdot\eta+\frac{\sqrt{2}\,\delta_{2s}}{\sqrt{s}}\norm{\h_{T_0^c}}_1\right)\norm{\h_{T_0\cup T_1}}_2,
\endaligned$$
which implies that
\begin{equation}\label{normhT01}	
	\norm{\h_{T_0\cup T_1}}_2\leq \frac{2\sqrt{1+\delta_{2s}}}{1-\delta_{2s}}\cdot\eta+\frac{\sqrt{2}\,\delta_{2s}}{1-\delta_{2s}}\cdot\frac{\norm{\h_{T_0^c}}_1}{\sqrt{s}}. 
\end{equation}
This, together with (\ref{normhT0c}), gives the~following estimate 
$$ \norm{\h_{T_0\cup T_1}}_2\leq \alpha\cdot\eta+\beta\cdot\frac{\norm{\h_{T_0}}_1}{\sqrt{s}}+2\beta\cdot\frac{\norm{\x_{T_0^c}}_1}{\sqrt{s}}, $$
where
$$ \alpha=\frac{2\sqrt{1+\delta_{2s}}}{1-\delta_{2s}}\qquad \textrm{and}\qquad \beta=\frac{\sqrt{2}\,\delta_{2s}}{1-\delta_{2s}}. $$
Since $\norm{\h_{T_0}}_1\leq\sqrt{s}\norm{\h_{T_0}}_2\leq\sqrt{s}\norm{\h_{T_0\cup T_1}}_2$, therefore
$$ \norm{\h_{T_0\cup T_1}}_2\leq \alpha\cdot\eta+\beta\cdot\norm{\h_{T_0\cup T_1}}_2+2\beta\cdot\epsilon,$$
where recall that $\epsilon=\frac{1}{\sqrt{s}}\norm{\h_{T_0^c}}_2$, hence
\begin{equation}\label{ineq2}
	\norm{\h_{T_0\cup T_1}}_2\leq\frac{1}{1-\beta}\left(\alpha\cdot\eta+2\beta\cdot\epsilon\right),
\end{equation}
as long as $\beta<1$ which is equivalent to $\delta_{2s}<\sqrt{2}-1$. 

Finally, (\ref{ineq1}) and (\ref{ineq2}) imply the main result
$$\aligned
	\norm{\h}_2&\leq\norm{\h_{T_0\cup T_1}}_2+\norm{\h_{(T_0\cup T_1)^c}}_2\leq \norm{\h_{T_0\cup T_1}}_2+\norm{\h_{T_0}}_2+2\epsilon \\
		&\leq 2\norm{\h_{T_0\cup T_1}}_2+2\epsilon\leq \frac{2\alpha}{1-\beta}\cdot\eta+\left(\frac{4\beta}{1-\beta}+2\right)\cdot\epsilon
\endaligned $$
and the constants in the~statement of the~theorem equal
$$ C_0=\frac{4\beta}{1-\beta}+2=2\frac{1+\beta}{1-\beta}=2\frac{1+\left(\sqrt{2}-1\right)\delta_{2s}}{1-\left(\sqrt{2}+1\right)\delta_{2s}} \quad \textrm{and}\quad 
		C_1=\frac{2\alpha}{1-\beta}=\frac{4\sqrt{1+\delta_{2s}}}{1-\left(\sqrt{2}+1\right)\delta_{2s}}.$$
\end{proof}

\section{Stable reconstruction from exact data} 

In this section we will assume that our observables are exact, i.e.
$$ \y=\Fi\x, \quad\textrm{where}\quad \x\in\H^n, \; \Fi\in\H^{m\times{n}}, \; \y\in\H^m.$$
The~undermentioned result is a~natural corollary of~Theorem~\ref{l1minthm} for $\eta=0$.

\begin{corollary}\label{l1mincor}
	Let $\Fi\in\H^{m\times n}$ satisfies the~$2s$-restricted isometry property with a~constant $\delta_{2s}<\sqrt{2}-1$. Then for any $\x\in\H^n$ and $\y=\Fi\x\in\H^m$, the~solution $\x^{\#}$ of the~problem
\begin{equation}\label{l1minexact}
\argmin\limits_{\z\in\H^n} \norm{\z}_1 \quad\text{subject to}\quad \Fi\z=\y
\end{equation}
satisfies 
\begin{equation} \label{exact-ineq1}
\norm{\x^{\#}-\x}_1 \leq C_0\norm{\x-\x_s}_1
\end{equation} and
\begin{equation} \label{exact-ineq2}
\norm{\x^{\#}-\x}_2 \leq \frac{C_0}{\sqrt{s}}\norm{\x-\x_s}_1 
\end{equation} 
with constant $C_0$ as in the~Theorem~\ref{l1minthm}. In particular, if $\x$ is $s$-sparse and there is no noise, then the~reconstruction by $\ell_1$-norm minimization is exact.
\end{corollary}

\begin{proof}
	Inequality (\ref{exact-ineq2}) follows directly from Theorem~\ref{l1minthm} for $\eta=0$. The~result for sparse signals is obvious since in this case $\x=\x_s$. We only need to prove~(\ref{exact-ineq1}).
	
	We will use the same notation as in the~proof of Theorem~\ref{l1minthm}. Recall that
	$$ \norm{\h_{T_0}}_1\leq\sqrt{s}\norm{\h_{T_0}}_2\leq\sqrt{s}\norm{\h_{T_0\cup T_1}}_2 $$
	which together with (\ref{normhT01}) for $\eta=0$ implies
$$ \norm{\h_{T_0}}_1\leq \frac{\sqrt{2}\,\delta_{2s}}{1-\delta_{2s}}\cdot\norm{\h_{T_0^c}}_1.  $$	
Using this and (\ref{normhT0c}), deonoting again $\beta=\frac{\sqrt{2}\,\delta_{2s}}{1-\delta_{2s}}$, we get that
$$ \norm{\h_{T_0^c}}_1\leq \beta\norm{\h_{T_0^c}}_1+2\norm{\x_{T_0^c}}_1, \qquad \textrm{hence}\qquad \norm{\h_{T_0^c}}_1\leq \frac{2}{1-\beta}\norm{\x_{T_0^c}}_1. $$
Finally, we obtain the~following estimate on the~$\ell_1$ norm of the~vector $\h=\x^{\#}-\x$
$$ \norm{\h}_1=\norm{\h_{T_0}}_1+\norm{\h_{T_0^c}}_1\leq (1+\beta)\norm{\h_{T_0^c}}_1\leq \underbrace{2\frac{1+\beta}{1-\beta}}_{=C_0}\norm{\x_{T_0^c}}_1, $$
which finishes the proof.
\end{proof}

We conjecture that the~ requirement $\delta_{2s}<\sqrt{2}-1$ is not optimal -- there are known refinements of this condition for real signals (see e.g. {\cite[Chapter~6]{introCS}} for references). On the~other hand, the~authors of~\cite{cz} constructed examples of $s$-sparse real signals which can not be uniquely reconstructed via $\ell_1$-norm minimization for $\delta_s>\frac13$. This gives an~obvious upper bound for~$\delta_{s}$ also for the~general quaternion case.

\section{Numerical experiment}

In~\cite{bb} we presented results of numerical experiments of sparse quaternion vector~$\x$ reconstruction (by $\ell_1$-norm minimization) from its linear measurements $\y=\Fi\x$ for the~case of  real-valued measurement matrix~$\Fi$. Those experiments were inspired by the~articles~\cite{barthelemy2015,hawes2014,l1minqs} and involved expressing the $\ell_1$ quaternion norm minimization problem in terms of the~second-order cone programming (SOCP).

In view of the~main results of this paper (Theorem~\ref{l1minthm}, Corollary~\ref{l1mincor}) and having in mind that quaternion Gaussian random matrices satisfy (with overwhelming probability) the restricted isometry property~\cite{bbRIP}, we performed similar experiments for the~case of~quaternion matrix~$\Fi\in\H^{m\times n}$ -- as in~\cite{l1minqs}. By a~quaternion Gaussian random  matrix we mean a~matrix $\Fi=(\Ffi_{ij})\in\H^{m\times n}$ whose entries $\Ffi_{ij}$ are independent quaternion Gaussian random variables with distribution denoted by $\norml_{\H}\left(0,\sigma^2\right)$, i.e. 
$$ \Ffi_{ij} = \Ffi_{\mathbf{r},ij} + \Ffi_{\i,ij}\i + \Ffi_{\j,ij}\j + \Ffi_{\k,ij}\k, $$ 
where 
$$ \Ffi_{e,ij}\sim\norml\left(0,\frac{\sigma^2}{4}\right), \; e\in\{\mathbf{r},\i,\j,\k\}, \quad\text{and}\quad \Ffi_{e,ij} \text{ are pairwise independent.}$$
In particular, each $\Ffi_{ij}$ has independent components, which are real Gaussian random variables.
In what follows, we consider only the case of noiseless measurements, i.e. we solve the problem~\eqref{l1minexact}. 

Recall after~\cite{l1minqs} that problem~\eqref{l1minexact} is equivalent to 
\begin{align} \label{eq:4_4_02}
	\argmin\limits_{t\in\R_+} t\quad\text{subject to}\quad \y=\Fi\z,\,\norm{\z}_1\leq t.
\end{align} 
We decompose vectors $\y\in\H^m$ and $\z\in\H^n$ into real vectors representing their~real parts and components of their~imaginary parts
\begin{align*}
	\y = \y_{\mathbf{r}}+\y_{\i}\i + \y_{\j}\j + \y_{\k}\k,\qquad  \z = \z_{\mathbf{r}}+\z_{\i}\i + \z_{\j}\j + \z_{\k}\k,
\end{align*} 
where $\y_{\mathbf{r}},\y_{\i},\y_{\j},\y_{\k}\in\R^m$, $\z_{\mathbf{r}},\z_{\i},\z_{\j},\z_{\k}\in\R^n$. Denote 
\begin{align*}
	\z_{\mathbf{r}} = (z_{\mathbf{r},1},\ldots,z_{\mathbf{r},n})^T, \,\,\,
	\z_{\i} = (z_{\i,1},\ldots,z_{\i,n})^T, \,\,\,
	\z_{\j} = (z_{\j,1},\ldots,z_{\j,n})^T, \,\,\,
	\z_{\k} = (z_{\k,1},\ldots,z_{\k,n})^T.
\end{align*} 
and let $\Ffi_k\in\H^m$, $k\in\{1,\ldots,n\}$, be the~$k$-th column of the~matrix $\Fi$. Again decompose as previously 
$$ \Ffi_k = \Ffi_{\mathbf{r},k}+\Ffi_{\i,k}\i + \Ffi_{\j,k}\j + \Ffi_{\k,k}\k, $$
where $\Ffi_{\mathbf{r},k},\Ffi_{\i,k},\Ffi_{\j,k},\Ffi_{\k,k}\in\R^m$. Note that the~second constraint in~\eqref{eq:4_4_02} can be written in the~form 
$$ 	\norm{(z_{\mathbf{r},k},z_{\i,k},z_{\j,k},z_{\k,k})^T}_2\leq t_k\quad\text{for }k\in\{1,\ldots,n\}, $$
where~$t_k$ are positive real numbers such that $\sum\limits_{k=1}^n t_k = t$. Having that, we can rewrite~\eqref{eq:4_4_02} in the~real-valued setup in the~following way: 
\begin{align}\label{eq:4_4_03}
	\argmin\limits_{\tilde{\z}\in\R^n} \C^T\tilde{\z} \quad\text{subject to}\quad	&\tilde{\y}=\tilde{\Fi}\tilde{\z} \\ 
		     					&\text{and}\quad	\norm{(z_{\mathbf{r},k},z_{\i,k},z_{\j,k},z_{\k,k})^T}_2\leq t_k\quad\text{for }k\in\{1,\ldots,n\}, \nonumber
\end{align} 
where 
\begin{align}
\tilde{\z} &= (t_1,z_{\mathbf{r},1},z_{\i,1},z_{\j,1},z_{\k,1},\ldots,t_n,z_{\mathbf{r},n},z_{\i,n},z_{\j,n},z_{\k,n})^T\in\R^{5n}, \label{eq:4_4_04} \\
\C &= (1,0,0,0,0,\ldots,1,0,0,0,0)^T\in\R^{5n}, \label{eq:4_4_05} \\
\tilde{\y} &= (\y_{\mathbf{r}}^T,\y_{\i}^T,\y_{\j}^T,\y_{\k}^T)^T\in\R^{4m},  \label{eq:4_4_06} \\
\tilde{\Fi} &= \left(\begin{array} {ccccccccccc}
\0 & \Ffi_{\mathbf{r},1} & -\Ffi_{\i,1} & -\Ffi_{\j,1} & -\Ffi_{\k,1} & \ldots &
\0 & \Ffi_{\mathbf{r},n} & -\Ffi_{\i,n} & -\Ffi_{\j,n} & -\Ffi_{\k,n} \\
\0 & \Ffi_{\i,1} & \Ffi_{\mathbf{r},1} & -\Ffi_{\k,1} & \Ffi_{\j,1} & \ldots &
\0 & \Ffi_{\i,n} & \Ffi_{\mathbf{r},n} & -\Ffi_{\k,n} & \Ffi_{\j,n} \\
\0 & \Ffi_{\j,1} & \Ffi_{\k,1} & \Ffi_{\mathbf{r},1} & -\Ffi_{\i,1} & \ldots &
\0 & \Ffi_{\j,n} & \Ffi_{\k,n} & \Ffi_{\mathbf{r},n} & -\Ffi_{\i,n} \\
\0 & \Ffi_{\k,1} & -\Ffi_{\j,1} & \Ffi_{\i,1} & \Ffi_{\mathbf{r},1} & \ldots &
\0 & \Ffi_{\k,n} & -\Ffi_{\j,n} & \Ffi_{\i,n} & \Ffi_{\mathbf{r},n} 
\end{array}\right) \label{eq:4_4_07}
\end{align} 
and~$\tilde{\Fi}\in\R^{4m\times 5n}$.

This is a~standard form of the~SOCP, which can be solved using the~SeDuMi toolbox for MATLAB~\cite{SeDuMi}. The~solution 
\begin{equation}
	\tilde{\x}^{\#} = \left(t_1,x_{\mathbf{r},1}^{\#},x_{\i,1}^{\#},x_{\j,1}^{\#},x_{\k,1}^{\#},\ldots,t_n,x_{\mathbf{r},n}^{\#},x_{\i,n}^{\#},x_{\j,n}^{\#},x_{\k,n}^{\#}\right)^T\in\R^{5n}
\end{equation} 
to the problem~\eqref{eq:4_4_03} can easily be expressed as 
\begin{equation} \label{eq:4_4_08}
	\x^{\#} = \left(x_{\mathbf{r},1}^{\#}+x_{\i,1}^{\#}\i+x_{\j,1}^{\#}\j+x_{\k,1}^{\#}\k, \; \ldots, \;x_{\mathbf{r},n}^{\#}+x_{\i,n}^{\#}\i+x_{\j,n}^{\#}\j+x_{\k,n}^{\#}\k\right)\in\H^n,
\end{equation} 
which is the~solution of our original problem~\eqref{l1minexact}.

\renewcommand{\theenumi}{\textbf{\arabic{enumi}}}

The~experiments were carried out in MATLAB R2016a on a~standard PC machine, with Intel(R) Core(TM) i7-4790 CPU (3.60GHz), 16GB RAM and with Microsoft Windows 10~Pro. The~algorithm consisted of the~following steps: 
\begin{enumerate}
	\item Fix constants $n=256$ (length of~$\x$) and~$m$ (number of measurements, i.e. length of~$\y$) and generate the~measurement matrix $\Fi\in\H^{m\times n}$ with Gaussian entries sampled from i.i.d. quaternion normal distribution $\norml_{\H}\left(0,\frac{1}{m}\right)$;
	\item Choose the~sparsity $s\leq \frac{m}{2}$ and draw the~support set $S\subseteq\{1,\ldots,n\}$ with $\# S=s$, uniformly at random. Generate a~vector $\x\in\H^n$ such that $\supp\x=S$ with i.i.d. quaternion normal distribution $\norml_{\H}(0,1)$;
	\item \label{alg:3} Compute $\y=\Fi\x\in\H^m$; 
	\item \label{alg:4} Construct vectors $\tilde{\y},\C$ and matrix $\tilde{\Fi}$ as in~\eqref{eq:4_4_04}--\eqref{eq:4_4_07};
	\item \label{alg:5} Call the~SeDuMi toolbox to solve the~problem~\eqref{eq:4_4_03} and calculate the~solution $\tilde{\x}^{\#}$;
	\item \label{alg:6} Compute the solution $\x^{\#}$ using~\eqref{eq:4_4_08} and the~errors of reconstruction (in the~$\ell_1$- and $\ell_2$-norm sense), i.e. $\norm{\x^{\#}-\x}_1$ and~$\norm{\x^{\#}-\x}_2$.
\end{enumerate}

The~experiment was carried out for $m=2,\ldots,64$ and $s=1,\ldots,\frac{m}{2}$. The range of $s$ is not accidental -- it is known in general that the~minimal number~$m$ of measurements needed for the~reconstruction of an~$s$-sparse vector is~$2s$~\cite[Theorem 2.13]{introCS}. For each pair of $(m,s)$ we performed 1000 experiments, saving the~errors of each reconstruction and the~number of perfect reconstructions (the~reconstruction is said to be perfect if $\norm{\x^{\#}-\x}_2\leq 10^{-7}$). For comparison we also repeated this experiment for the~case of $\Fi\in\R^{m\times n}$ and $\x\in\R^{n}$. The~percentage of perfect reconstructions in each case is presented in Fig.~\ref{fig:6_1} and Fig.~\ref{fig:6_2}~(a).

\begin{figure}[h]
	\centerline{
	\epsfig{file=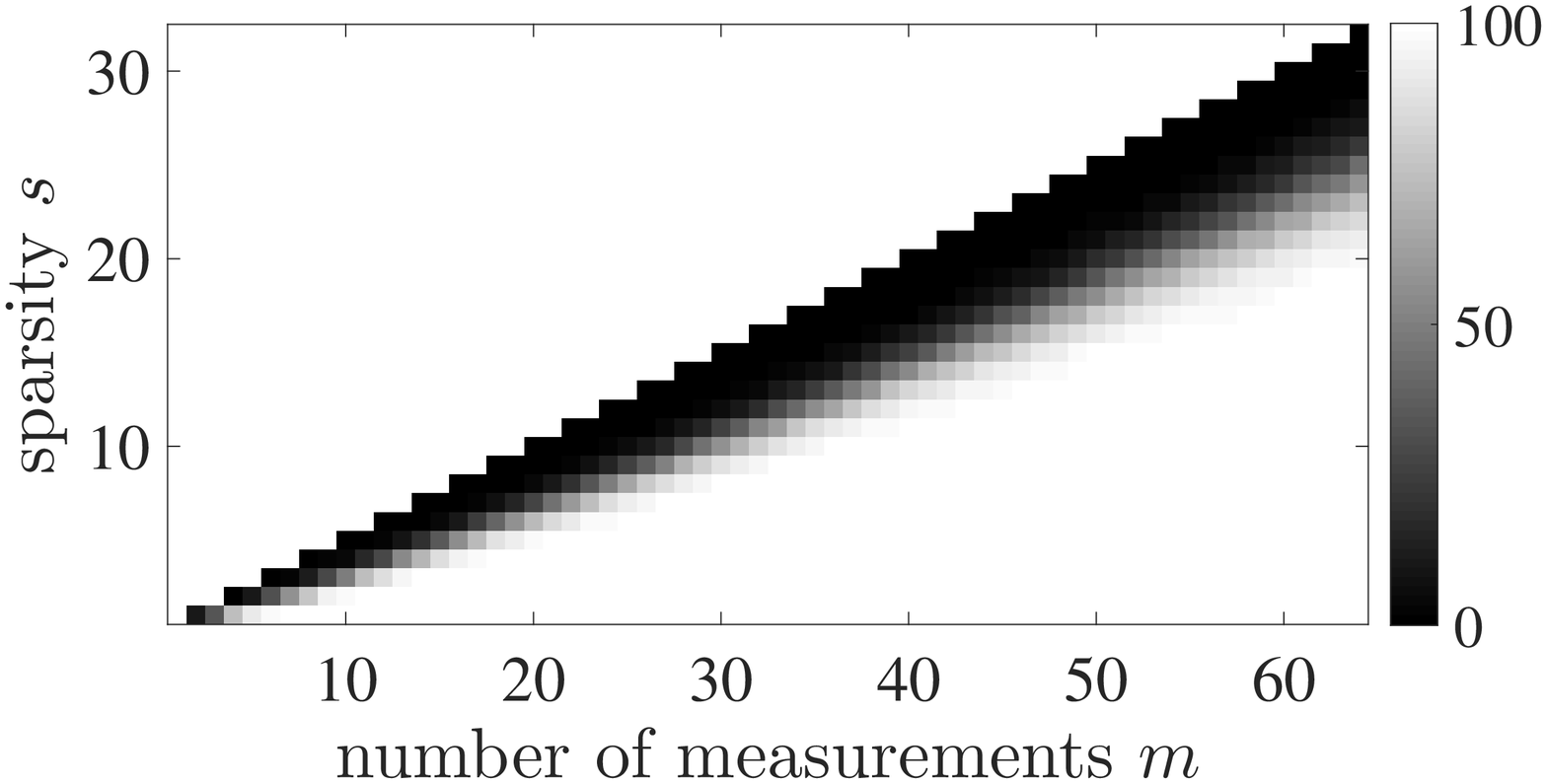,width=0.49\textwidth}
	\epsfig{file=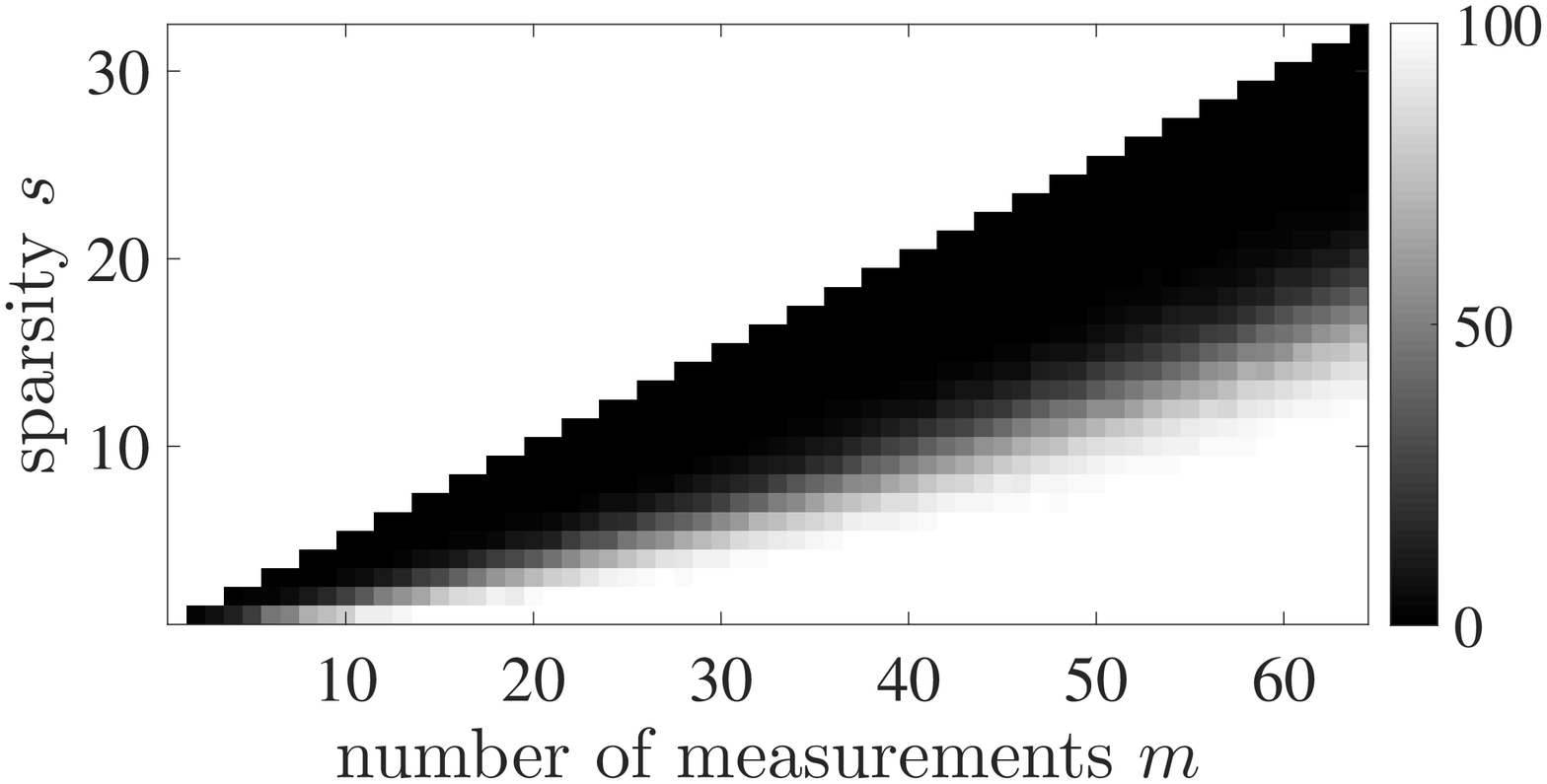,width=0.49\textwidth}
	}
	\begin{minipage}{0.49\textwidth} \centerline{(a) $\Fi\in\H^{m\times n}$} \end{minipage}
	\begin{minipage}{0.49\textwidth} \centerline{(b) $\Fi\in\R^{m\times n}$} \end{minipage}
	
	\caption{Results of the recovery experiment for $n=256$ and different $m$ and $s$. Image intensity stands for the~percentage of perfect reconstructions.}
	\label{fig:6_1}
\end{figure}

\begin{figure}[h]
	\centerline{
	\epsfig{file=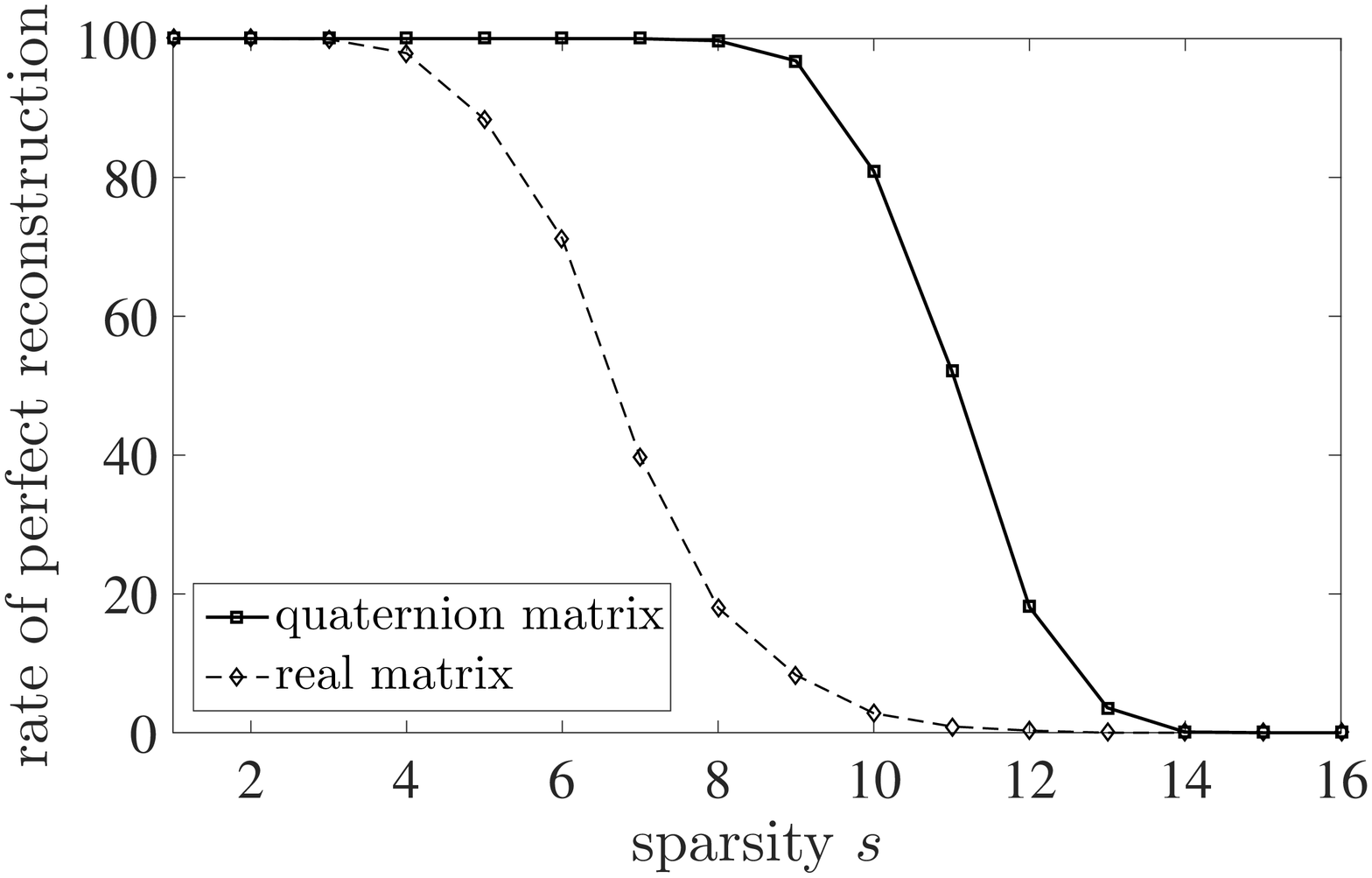,width=0.49\textwidth}
	\epsfig{file=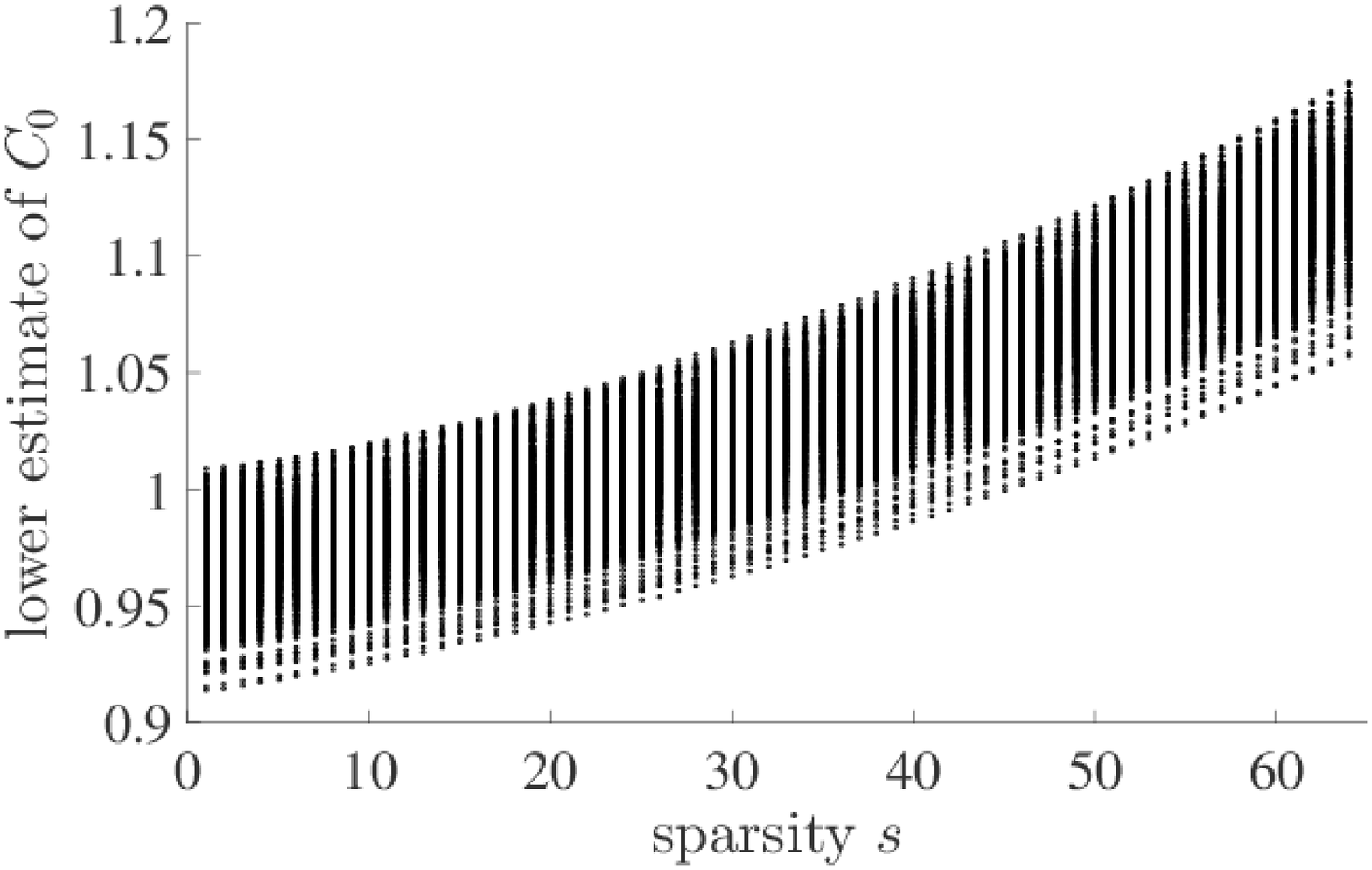,width=0.49\textwidth}
	}
	\begin{minipage}{0.49\textwidth} \centerline{(a)} \end{minipage}
	\begin{minipage}{0.49\textwidth} \centerline{(b)} \end{minipage}
	
	\caption{(a) Comparison of the~recovery experiment results for $n=256$, $m=32$ and different values of~$s$. (b) Lower estimate of the~constant $C_0$ in~Corollary~\ref{l1mincor} obtained from the~inequality~\eqref{exact-ineq1} for $n=256$ and $m=32$.}
	\label{fig:6_2}
\end{figure}

Fig.~\ref{fig:6_1}~(a) presents dependence of the~perfect recovery percentage on the~number of measurements~$m$ and sparsity~$s$ in the~quaternion case. We see that simulations confirm our theoretical~considerations. Fig.~\ref{fig:6_1}~(b) shows the~same results for the~real case, i.e. $\Fi\in\R^{m\times n}$ and $\x\in\R^n$. Note that in the~first experiment for $m=32$ and $s\leq 9$ the~recovery rate is greater than 95\%, the~same holds for $m=64$ and $s\leq 20$. It is also worth noticing that the~results for coresponding pairs $(m,s)$ are much better in the~quaternion setup than in the~real-valued case (see~Fig.~\ref{fig:6_2}(a)). We explain this phenomenon in~{\cite[Lemma~3.1]{bbRIP}}, namely, we show that for a~fixed vector~$\x\in\H^m$ and the~ensemble of quaternion Gaussian random matrices $\Fi\in\H^{m\times n}$, the~ratio random variable $\frac{\norm{\Fi\x}_2^2}{\norm{\x}_2^2}$ has distribution $\Gamma(2m,2m)$, i.e. its variance equals $\frac{1}{2m}$, which is four times smaller than in the~case of a~real vector and real Gaussian matrices of the~same size. In other words, a~quaternion Gaussian random matrix statistically has smaller restricted isometry constant than its real counterpart.

\medskip
We also performed another experiment illustrating the~approximated reconstruction of non-sparse quaternion vectors from the~exact data -- as stated in Corollary~\ref{l1mincor}. We fixed constants $n=256$ and $m=32$ and generated the measurement matrix $\Fi\in\H^{m\times n}$ with random entries sampled from i.i.d. quaternion normal distribution and $1000$ arbitrary vectors $\x\in\H^n$ with standard Gaussian random quaternion entries ($\sigma^2=1$), without assuming their sparsity. The~above-described algorithm (steps \ref{alg:3}.--\ref{alg:6}.) was applied to approximately reconstruct the~vectors. We used the~reconstruction errors $\norm{\x^{\#}-\x}_1$ to obtain a~lower bound on the~constant $C_0$ as a~function of~$s$, for $s=1,\ldots,64$, using inequality~\eqref{exact-ineq1}, i.e. 
$$ C_0\geq \frac{\norm{\x^{\#}-\x}_1}{\norm{\x-\x_s}_1}, $$
where $\x_s$ denotes the best $s$-sparse approximation of $\x$. The~results of this experiment are shown in Fig.~\ref{fig:6_2}~(b) in the~form of a~scatter plot -- each point represents a~lower estimate of~$C_0$ for one vector $\x$ and sparsity $s$. We see in particular that -- as expected -- the~dependence on $s$ is monotone.

\section{Conclusions}

The results of this article, together with aforementioned~\cite{bbRIP}, form a~theoretical background of the~classical compressed sensing methods in the~quaternion algebra. We extended the~fundamental result of this theory to the~full quaternion case, namely we proved that if a~quaternion measurement matrix satisfies the~RIP with a~sufficiently small constant, then it is possible to reconstruct sparse quaternion signals from a~small number of their measurements via $\ell_1$-norm minimization. We also estimated the~error of the~approximated reconstruction of a~non-sparse quaternion signal from exact and noisy data. This improves our previous result for real measurement matrices and sparse quaternion vectors~\cite{bb} and explains success of various numerical experiments in the~quaternion setup~\cite{barthelemy2015,hawes2014,l1minqs}.

There are several possibilities of further research in this field -- both in theoretical and applied directions. Among others:\\
-- further refinements of the~main results in the~quaternion algebra or their extensions to different algebraic structures,\\
-- search for other than Gaussian quaternion matrices satisfying the RIP,\\
-- adjust reconstruction algorithms to quaternions.\\
-- applications of the~theory in practice.\\
In view of numerous articles concerning quaternion signal processing published in the~last decade, we expect that this new branch of compressed sensing will attract attention of even more researchers and considerably develop.

\bigskip\bigskip
\noindent
\textbf{Acknowledgments}

The~research was supported in part by WUT grant No. 504/01861/1120. The~work conducted by the~second author
was supported by a scholarship from Fundacja Wspierania Rozwoju Radiokomunikacji i Technik Multimedialnych.

\end{document}